\documentclass[11pt]{article}

\usepackage{graphicx, mathtools, amssymb, latexsym, amsmath, amsfonts, amsthm, bbm, tikz}
\usepackage{array, fancyhdr}
\usepackage{comment, url}

\theoremstyle{plain}
\newtheorem{theorem}{Theorem}[section]
\newtheorem{lemma}[theorem]{Lemma}

\newcommand\HG{\mathrm{HG}}

\begin{document}
\title{Hat Guessing Numbers of Degenerate Graphs}
\author{Xiaoyu He\thanks{Department of Mathematics, Stanford University, Stanford, CA 94305.
Email: {\tt alkjash@stanford.edu}. Research supported by a NSF Graduate Research Fellowship grant number DGE-1656518.}
\and Ray Li\thanks{Department of Computer Science, Stanford University, Stanford, CA 94305.  Email: {\tt rayyli@cs.stanford.edu}. Research supported by a NSF Graduate Research Fellowship grant number DGE-1656518.}}
\date{\today}

\maketitle

\begin{abstract}
    Recently, Farnik asked whether the hat guessing number $\HG(G)$ of a graph $G$ could be bounded as a function of its degeneracy $d$, and Bosek, Dudek, Farnik, Grytczuk and Mazur showed that $\HG(G)\ge 2^d$ is possible. We show that for all $d\ge 1$ there exists a $d$-degenerate graph $G$ for which $\HG(G) \ge 2^{2^{d-1}}$. We also give a new general method for obtaining upper bounds on $\HG(G)$.   The question of whether $\HG(G)$ is bounded as a function of $d$ remains open.
\end{abstract}

\section{Introduction}

Hat puzzles have long been a mainstay of recreational mathematics, colorfully touching on ideas as diverse as the axiom of choice, modular arithmetic, and coding theory (see e.g.~\cite{W, W2}). In this note, we study a particular hat puzzle on graphs introduced by Butler, Hajiaghayi, Kleinberg, and Leighton \cite{BHKL}, which has attracted some recent interest in discrete mathematics~\cite{ABST, BDFGM, F, F2, G, GG, S}.

Suppose $n$ players are sitting on the vertices of a (finite, simple) graph $G$. An adversary puts a colored hat on each of their heads, in one of $q$ colors. The players can only see the hats on their neighbors' heads, and in particular no player sees their own hat. The players simultaneously guess the colors of their own hats, and they collectively win if any single one of them guesses correctly. The players may not communicate after the hat colors are assigned but may agree upon a strategy beforehand. The {\it hat guessing number} $\HG(G)$ of $G$ is then the largest $q$ for which the players have a winning strategy in the game with $q$ colors.

To our knowledge, the only connected graphs $G$ for which we know the exact value of $\HG(G)$ are complete graphs, trees, or pseudotrees (connected graphs with exactly one cycle). 
The classic case is when $G=K_n$ is the complete graph on $n$ vertices, and in this case $\HG(K_n)=n$. One winning strategy is for the $i$-th player to guess the hat color (identifying colors with $\mathbb Z/n \mathbb Z$) that would make the total of all the colors sum to $i \pmod{n}$. 
Butler, Hajiaghayi, Kleinberg, and Leighton~\cite{BHKL} determined that $\HG(T)=2$ for all trees $T$, and Szczechla~\cite{S} showed that for the cycle $C_n$ of length $n$, $\HG(C_n) = 3$ if $n=4$ or $n$ is a multiple of $3$, and $\HG(C_n) = 2$ otherwise.
Kokhas and Latyshev \cite{KL} showed that in fact the only connected graphs $G$ with $\HG(G)=2$ are trees and graphs with a single cycle whose length is not 4 or a multiple of 3.

Farnik \cite{F2} observed that if $G$ has maximum degree $\Delta$, then $\HG(G)\le e\Delta$ by the Lov\'asz Local Lemma, and asked whether $\HG(G)$ is also bounded as a function of its degeneracy $d$ (which is always at most $\Delta$). 
Recall that a graph $G$ is $d$-\emph{degenerate} if there is a left-to-right ordering of its vertices where each vertex has at most $d$ neighbors to its left.
Farnik showed that if the players strategies are constrained to have a so-called ``bi-polar'' property, then the players cannot win unless there $q\le d+1$ colors, and conjectured that $\HG(G)\le d+1$ in general.
This was later disproved by Bosek, Dudek, Farnik, Grytczuk, and Mazur \cite{BDFGM}, who showed that for $n$ large enough, $\HG(B_{d,n}) \ge 2^{d}$ for the so-called \emph{book} graph $B_{d,n}$, which consists of a $d$-clique complete to $n$ isolated vertices, and which is $d$-degenerate.\footnote{\cite{BDFGM,F2} refer to $d+1$ as the \emph{coloring number} of the graph.}
Alon, Ben-Eliezer, Shangguan, and T\'amo~\cite{ABST} considered another variant of this question, showing that if $q$ is a prime power and the players' strategies are constrained to be linear (identifying the $q$ hat colors with elements of $\mathbb{F}_q$), then the players cannot win unless $q\le d+1$.

Our first main result proves that in fact $\HG(G)$ can be at least double exponential in terms of the degeneracy $d$ of $G$. 
Furthermore, we determine the hat guessing number exactly for a certain family of $d$-degenerate graphs $G$, adding an infinite family to the small list of graphs whose hat guessing number is known precisely.

Let $T_d(N)$ denote the complete rooted $N$-ary tree with depth $d$, and define $G_d(N)$ to be the graph obtained from $T_d(N)$ by drawing an edge between every vertex and each of its ancestors (if the edge does not exist already).
Ordering the vertices by their depth in $T_d(N)$, we see that $G_d(N)$ is $d$-degenerate.
Recall that \emph{Sylvester's sequence} $(s_n)_{n=0}^\infty$, defined by $s_0=2$ and 
\[
s_n = s_0 \cdots s_{n-1} + 1 = s_{n-1}^2-s_{n-1}+1
\]
for $n\ge 1$, satisfies $s_n = \lfloor E^{2^{n+1}} +\frac{1}{2}\rfloor$ for some $E \approx 1.264$.

\begin{theorem}
\label{thm:main}
For any $d\ge 1$ and all $N$ sufficiently large in terms of $d$,
\[
\HG(G_d(N)) = s_{d}-1.
\]
\end{theorem}

We say that a vertex ordering of a graph $G$ is \emph{$d$-degenerate} if each vertex has at most $d$ neighbors to the left, and that such an ordering has {\it depth} $D$ if the longest left-to-right path contains $D$ edges. Our second main result shows that Theorem~\ref{thm:main} is close to best possible, in the sense that $\HG(G)$ is always doubly-exponentially bounded in terms of $d$ and $D$.

\begin{theorem}
\label{thm:upper}
If $d\ge 2$, $D\ge 1$ and $G$ has a $d$-degenerate vertex ordering of depth $D$, then
\[
\HG(G) < 2^{d^{D+1}}.
\]
\end{theorem}

Our last result shows that $\HG(B_{d,n})=d^{(1+o(1))d}$ for $n$ sufficiently large, improving the bounds in \cite{BDFGM} and characterizing the asymptotic growth of $\HG(B_{d,n})$ up to a lower order term in the exponent.
\begin{theorem}
\label{thm:book}
For all $d\ge 1$ and $n$ sufficiently large in terms of $d$, we have
\begin{align}
  (d+1)!\le \HG(B_{d,n}) \le d^{d-2}(d^2+d-1)+1.
\end{align}
\end{theorem}

In the next section we prove the upper bounds in Theorems~\ref{thm:main} and \ref{thm:upper}; the main innovation is a deterministic adversary strategy that precommits to using very few colors on certain vertices of $G$. The lower bound in Theorem~\ref{thm:main} is then proved separately in Section~\ref{sec:lower}; we build a guessing strategy starting with the largest depth and show that, from depth $i+1$ to depth $i$, we can roughly square-root the number of candidate colorings.
We finish by proving Theorem~\ref{thm:book} in Section~\ref{sec:book}.

\paragraph{Related work}
Perhaps the most famous hat guessing puzzle was introduced by Ebert \cite{E}.
Each player is independently given a red or blue hat with equal probability and the players, seeing the other players' hats, simultaneously either guess their hat color or pass.
They win if at least one player guesses correctly and no player guesses incorrectly.
The 3 player version was popularized in the New York Times \cite{R2}.
Many variations of this puzzle have been studied (see \cite{F2, K} for surveys), including one \cite{BHKL} where, like in this work, the sight graph is an arbitrary graph, rather than a clique.
Hat guessing puzzles have found connections to coding theory \cite{EMV}, to auctions \cite{AFGHIS}, to network coding \cite{GR,R}, to finite dynamical systems \cite{G}, and possibly to understanding DNA \cite{GKRT}.

\section{Upper bounds}
We prove a general upper bound on $\HG(G)$ that implies both Theorem~\ref{thm:upper} and the upper bound in Theorem~\ref{thm:main}.

If $G$ is a graph with vertices ordered $v_1,\ldots, v_n$, define $A(v_i)\coloneqq N(v)\cap \{v_1,\ldots, v_{i-1}\}$ to be the set of neighbors of $v_i$ to the left of $v_i$  ($A$ for ``ancestor'').
Let $[n]$ be the set $\{1,\dots,n\}$, and for a set $S$, let $\binom{S}{k}$ denote the collection of $k$-element subsets of $S$.

\begin{lemma}\label{lem:upper}
  Let $G$ be a graph with vertices ordered $v_1,\ldots, v_n$, and define $t_i$ for $i=1,\ldots, n$ recursively by
  \[
  t_i \coloneqq 1 + \prod_{v_j \in A(v_i)} t_j,
  \]
  where the empty product is $1$. Then, $\HG(G) < \max\{t_1,\ldots, t_n\}$.
\end{lemma}
\begin{proof}
    Fix any guessing strategy for the players, which we may assume the adversary knows. We describe an adversary strategy which prevents any vertex from guessing correctly using $q$ colors, where $q = \max\{t_1,\ldots, t_n\}$.

  The adversary pre-commits to using only a color from $[t_i]$ for the hat of $v_i$.
  The adversary then chooses the hat colors of $v_i$ in decreasing order of $i$, starting from the rightmost vertex $v_n$.
  Assume that, for some $1\le i\le n$, all vertices to the right of $v_i$ have been assigned hat colors.
  Over all possible remaining ways to place the hats, there are at most
  \[
  \prod_{j\in A(v_i)} t_j = t_i-1
  \]
  distinct tuples of colors that can appear on the set $A(v_i)$. Since the hat colors to the right of $v_i$ have all been assigned, the guess of $v_i$ now depends solely on the colors in $A(v_i)$, so at most $t_i-1$ colors are possible for $v_i$ to guess. Hence, the adversary can find a color out of $[t_i]$ that $v_i$ is guaranteed not to guess.
  
  We have shown that there exists a color $q_i\in[t_i]$  \emph{not} guessed by $v_i$, regardless of the hat assignments yet to be made. Giving $v_i$ this hat color and continuing to the left, we are done.
\end{proof}

We now can deduce the upper bounds simply by exhibiting the appropriate vertex orderings.

\begin{proof}[Proof of Theorem~\ref{thm:upper}]
Suppose $G$ has a vertex ordering where each vertex has at most $d$ left-neighbors, and the longest left-to-right path has length $D$. We say that $v_i$ has depth $k$ if the longest left-to-right path with $v_i$ as its rightmost vertex has length $k$. Applying Lemma~\ref{lem:upper} to this ordering, we find that if $a_k$ satisfies $a_0 = 2$ and $a_k = a_{k-1}^d + 1$ for $k >0$, then $t_i \le a_k$ if $k$ is the depth of $v_i$, and so $\HG(G) < a_D$. It is easy to check that $a_D \le 2^{d^{D+1}}$, as desired.
\end{proof}

The upper bound in Theorem~\ref{thm:main} is similar.

\begin{proof}[Proof of upper bound in Theorem~\ref{thm:main}]
 Order the vertices of the graph $G_d(N)$ so that every vertex comes after its parent in the tree $T_d(N)$, so that for each $v\in G_d(N)$, the set $A(v)$ is exactly the set of ancestors of $v$ in $T_d(N)$.
 By the definition of $G_d(N)$, each vertex at depth $k$ has exactly one ancestor of every depth smaller than $k$. Applying Lemma~\ref{lem:upper} to this ordering, we see that $t_i = s_k$ where $k$ is the depth of $v_i$ and $s_k$ is the $k$-th term of Sylvester's sequence. The depth is $d$, so $\HG(G)<s_d$, as desired.
 \end{proof}
 
\section{The lower bound}\label{sec:lower}
 
  We now show the lower bound $\HG(G_d(N))\ge s_d-1$ for sufficiently large $N$. Our guessing strategy uses repeatedly the simple observation that each player can act under the assumption that every other player guesses incorrectly, since otherwise they would collectively win.
  
  \begin{proof}[Proof of the lower bound in Theorem~\ref{thm:main}.]
  Let $q = s_d - 1$ and $N\ge q^{dq}$.
  For $i=0,1,\dots,d-1$, let $r_i=s_{i+1}-2$ and let $r_d = q^{d+1}$, so that in particular $r_{d-1} = q -1$.
  Order the family $[q]^*$ of all tuples over $[q]$ by length, with ties broken lexicographically. Order the vertices of $G = G_d(N)$ so that each vertex comes to the right of its parent in $T_d(N)$.
  
  We say $v\in V(G)$ is a child (parent, descendant, ancestor) of $w\in V(G)$ if and only if it is a child (parent, descendant, ancestor) of $w$ in $T_d(N)$.
  Let $D(v)$ denote the set of descendants of $v$, let $A(v)$ denote the set of ancestors of $v$, and let $\bar A(v)=A(v)\cup \{v\}$.
  If $\chi$ is an assignment of hat colors to $V(G)$ and $U\subset V(G)$, let $\chi(U)\in [q]^{|U|}$ be the tuple of hat colors of elements of $U$, ordered according to the vertex ordering of $G$. 
  
  For each $v\in V(G)$ of depth $0\le i\le d-1$, assign each of its children $w$ a subset $S_w\subset[q]^{i+1}$ of size $r_i+1$ such that all $\binom{q^{i+1}}{r_i+1}$ such subsets are assigned to at least one child of $v$.
  This is possible as the number of children of $v$ is $N \ge \binom{q^{i+1}}{r_i+1}$ for all $i\le d-1$.

  We are now ready to describe the guessing strategy. Roughly speaking, the idea is that the vertices of larger depth in the tree are so numerous that the vertices of smaller depth can deduce a lot of information about their own colors from the assumption that their descendants all guess incorrectly. The strategy will be built starting from the largest depth $d$, so that we first fix the guess of $w$ before the guess of any element of $A(w)$. Crucially, by the definition of $G$, all information available to $w$ is available to any $v\in A(w)$ except for the color of $v$.
  
  We prove that, for any $0\le i\le d$ and any depth-$i$ vertex $v$, the vertices of $D(v)$ can guess their hat colors so that vertex $v$ (and its ancestors) can deduce a set of $r_i$ possibilities for $\chi(\bar A(v))$, using solely the information $\chi(D(v))$. 
  Formally, there is a guessing strategy for vertices of depth $>i$ and a functions $f_v:[q]^{D(v)} \to \binom{[q]^{i+1}}{r_i}$ depending only on $\chi(D(v))$, such that if all vertices of depth greater than $i$ guess wrong, then $f_{v}(\chi(D(v))$ contains $\chi(\bar A(v))$.
 
  We proceed by reverse induction on $i$.
  The base case $i=d$ is trivial, as for every vertex $v$ of depth $d$, we can take $f_v(\cdot)=[q]^{d+1}$, which indeed is a set of size $r_d$ and contains every possible hat assignment for $\bar A(v)$.

  Now suppose $i\le d-1$ and the assertion is true for $i+1$, so that we have built the guessing strategy for vertices of depth greater than $i+1$, and there exist functions $f_w$ for vertices $w$ at depth $i+1$.
  For each such $w$, let $R_w\coloneqq f_w(\chi(D(w)))$ be the set of $r_{i+1}$ assignments that $w$ can ``deduce'' for $\chi(\bar A(w))$.
  Say that a color $c \in [q]$ is \emph{$w$-abundant} if there are at least $r_i+1$ colorings $\chi' \in R_w$ with $\chi'(w)=c$, and let $C_w$ be the set of $w$-abundant colors.
  
  If $i = d-1$, then there are clearly at most $q=r_i+1$ many $w$-abundant colors.
  If $i < d-1$, 
  \[
  |R_w| = r_{i+1} = r_i^2+3r_i+1 < (r_i + 1)(r_i+2),
  \]
  so by the pigeonhole principle, there are again at most $r_i+1$ $w$-abundant colors.
  Thus, we may define $\phi_w:C_w\to S_w$ mapping the $k$th smallest $w$-abundant color to $k$th smallest element of $S_w$ for all $k$.
  Note that vertex $w$ can compute $\phi_w$: $w$ sees all of $D(w)$, so $w$ can compute $R_w$, and thus $C_w$, and thus $\phi_w$ as $S_w$ was predetermined.
  If $\chi(\bar A(v)) = \phi_w(c)$ for some $w$-abundant color $c$, vertex $w$ guesses $c$. Otherwise, $w$ guesses arbitrarily.

  Let $v$ be a vertex of depth $i$.
  We show that, from the hat colors of $D(v)$, we can compute a set $f_v(\chi(D(v))$ of $r_i$ hat assignments to $\bar A(v)$ that contains the correct assignment if all vertices in $D(v)$ guess incorrectly.
  First, from $\chi(D(v))$, for each child $w$ of $v$, one can compute the set of $w$-abundant colors $C_w$ and the injection $\phi_w$, and hence can determine if $v$ has a child $w$ for which $\chi(w)$ \emph{not} $w$-abundant.
  
  We break into two cases based on whether $c_w = \chi(w)$ is $w$-abundant for all children $w$ of $v$. If there exists a single $w$ for which $c_w$ is not $w$-abundant, then there are at most $r_i$ colorings $\chi' \in R_w$ for which $\chi'(w)=c_w$, and $v$ can just let $f_v(\chi(D(v))$ be this set of at-most-$r_i$ colorings, restricted to $\bar A(v)$. 
  If all descendants of $w$ guess incorrectly, $R_w$ contains the correct coloring of $\bar A(w)$, so $f_v(\chi(D(v))$ contains the correct coloring of $\bar A(v)$.

  Otherwise, all children $w$ of $v$ have a hat color that is $w$-abundant.
  Call an assignment to $\bar A(v)$ \emph{good} if it is \emph{not} equal to $\phi_w(\chi(w))$ for any child $w$.
  Since the sets $S_w$ range over all $(r_i+1)$-subsets of $[q]^{i+1}$, and $\phi_w(\chi(w))\in S_w$ always, there are at most $r_i$ good assignments.
  Let $f_v(\chi(D(v)))$ be the set of good assignments, with some arbitrary assignments thrown in to make it size exactly $r_i$.
  Suppose all vertices in $D(v)$ guess incorrectly.
  We show the assignment to $\bar A(v)$ is good.
  For each $w$, either $\chi(\bar A(v))$ is not in the image of $\phi_w$, or $w$ guesses the color corresponding to $\chi(\bar A(v))$, which cannot be $\chi(w)$ since we assume $w$ guesses incorrectly. 
  In either case, $\chi(\bar A(v))\neq \phi_w(\chi(w))$ for all children $w$ of $v$, so the assignment to $\bar A(v)$ is good, and therefore in $f_v(\chi(D(v)))$.
  This shows how to construct $f_v(\cdot)$ in all cases, completing the induction.
  
  By the induction, it follows that the root vertex $v_0$ at depth $0$ can determine a set $f_{v_0}(\chi(D(v_0)))$ of $r_0 = s_1 - 2 = 1$ colorings for $\bar A(v_0) = \{v_0\}$, and so $v_0$ uniquely determines its hat color assuming all the other vertices guess incorrectly. This completes the proof.
\end{proof}

\section{Books}
\label{sec:book}
In this section, we prove Theorem~\ref{thm:book}. Recall that the \emph{book graph} $B_{d,n}$ is obtained by removing an $n$-clique from a complete graph $K_{n+d}$. The vertices of the removed $n$-clique are called the {\it pages} of the book, and the remaining $d$ vertices are its {\it spine}. Note that $B_{d,n}$ is $d$-degenerate. 

We first reduce $\HG(B_{d,n})$ to an equivalent geometric problem, introduced by Bosek, Dudek, Farnik, Grytczuk, and Mazur~\cite{BDFGM}. Let $h(\mathbb{N}^d)$ (which is $\mu_a(K_d)$ in the notation of~\cite{BDFGM}) denote the largest $t$ such that every $t$-subset of $\mathbb{N}^d$ can be covered by picking at most one point from every axis-aligned line. Formally, we say that {\it a set $S\subseteq \mathbb{N}^d$ is coverable} if there exists a partition $S=S_1 \sqcup S_2 \sqcup \cdots \sqcup S_d$ such that $S_i$ contains at most one point along any line parallel to the $i$-th coordinate axis. For example, $h(\mathbb{N}^2) = 5$ because any $5$-set in the plane is coverable, but the $6$-set $\{0,1\} \times \{0, 1, 2\}$ is not.
A similar concept was also considered e.g. in \cite{BG}.

\begin{lemma} \label{lem:book-reduction}
For any $d\ge 1$ and $n$ sufficiently large in terms of $d$,
\[
\HG(B_{d,n}) = h(\mathbb{N}^d) + 1.
\]
\end{lemma}
\begin{proof}
In~\cite[Theorem 13]{BDFGM}, it was shown that $\HG(B_{d,n})\ge h(\mathbb{N}^d) + 1$ for all $n$ large enough.
We show $\HG(B_{d,n})\le h(\mathbb{N}^d)+1$ for all $n$. Let $U$ denote the spine of $B_{d,n}$ and $V$ denote its pages, so that $U$ forms a $d$-clique complete to the $n$ pages $V$, which form an independent set. Let $S\in \mathbb{N}^d$ be a set of size $q=h(\mathbb{N}^d) + 1$ that is not coverable. Without loss of generality, $S$ is a subset of $[q]^d$, so we can also view elements of $S$ as hat colorings of $U$. In fact the adversary will pre-commit to using only elements of $S$ as the coloring on $U$.

  Fix a guessing strategy on $B_{d,n}$ with colors $[q+1]$, and for any $s\in S$ and $v\in V$ let $f_v(s)$ denote the guess that vertex $v$ makes if $s$ is the coloring of $U$. Since $|S| = q$, by the pigeonhole principle there is some color $q_v \in [q+1]$ that never appears in $\{f_v(s) : s\in S\}$.
  Give each $v\in V$ hat color $q_v$; this guarantees that no $v\in V$ guesses correctly.
  For each $u\in U$, the subset $S_u\subseteq S$ of colorings of $U$ where $u$ guesses correctly contains at most one point along any line parallel to $u$'s coordinate axis in $\mathbb{N}^d$.
  By the definition of $S$, there must exist some $s\in S$ outside all of the $S_u$. Thus there is a hat assignment where no vertex guesses correctly, and the upper bound follows.
\end{proof}

A lower bound $h(\mathbb{N}^d)\ge 2^d - 1$ was shown in~\cite{BDFGM}. We give a stronger lower bound and an upper bound that matches it up to $o(1)$ in the exponent. 

\begin{lemma} For all $d\ge 1$,
 \[
 (d+1)! - 1 \le h(\mathbb{N}^d)\le d^{d-2}(d^2+d-1).
 \]
\end{lemma}
\begin{proof}
  The upper bound follows from the observation that any set 
  \[
  S \subset [d]^{d-1}\times[d+1].
  \]
  of size $|S|=d^{d-2}(d^2+d-1)+1 < d^{d-1}(d+1)$ is not coverable. Indeed, at most $d^{d-2}(d+1)$ points can be picked along lines in any of the first $d-1$ axis directions, and at most $d^{d-1}$ points can be picked in the last axis direction for a total of at most $(d-1)d^{d-2}(d+1)+ d^{d-1} = |S|-1$ points covered, so at least one point of $S$ remains uncovered.
  
  For the lower bound, as $h(\mathbb{N})=1$, it suffices to prove 
  \[
  h(\mathbb{N}^d)\ge (d+1)h(\mathbb{N}^{d-1}) + d.
  \]
  Let $S\subset\mathbb{N}^d$ have size $(d+1)h(\mathbb{N}^{d-1}) + d$ and define $P_{i,y}\coloneqq \{x\in \mathbb{N}^d: x_i=y\}$ to be a hyperplane orthogonal to the $x_i$-axis.
  Say $y\in\mathbb{N}$ is \emph{$i$-abundant} if $|S\cap P_{i,y}| \ge h(\mathbb{N}^{d-1}) + 1$ (vacuously, no value is 0-abundant), and that $x\in \mathbb{N}^d$ is \emph{type-$i$} if $i$ is the smallest index such that $x_i$ is not $i$-abundant, and say $x$ is type-0 if it is not type-$i$ for any $i\in[d]$.
  By definition, for any $x$, there is exactly one $i\ge 0$ such that $x$ is type-$i$.
  By the pigeonhole principle, for any $i$, at most $d$ integers are $i$-abundant.
  Without loss of generality, these integers are $1,2,\dots,d_i\le d$.
  
  Let $L_{0,0}$ denote the set of type-0 points.
  Let $L_{i,y}\subseteq P_{i, y}$ denote the set of type-$i$ points $x$ such that $x_i=y$.
  Clearly the sets $L_{i,y}$ partition $\mathbb{N}^d$.
  For any $i,y$ with $L_{i,y}$ defined, let $\mathcal{L}_{i,y}$ denote the set of lines through a point in $L_{i,y}$ parallel to some $j$-axis with $j\neq i$.
  One can check that the $\mathcal{L}_{i,y}$ are pairwise disjoint: for any $i<i'$ and $y$ and $y'$, a line in both $\mathcal{L}_{i,y}$ and $\mathcal{L}_{i',y'}$ must be parallel to some $j$ axis for $j\neq i$ and $j\neq i'$ and furthermore must pass through points $x\in L_{i,y}$ and $x' \in L_{i',y'}$.
  Then $x_{i'}$ is not $i'$-abundant, so $x$ is not type-0, contradicting the fact that $x\in L_{i,y}$ if $i=0$.
  Also, $x'_i$ is not $i$-abundant, so if $i\neq 0$, then $x'$ is not type-$i'$ as $0<i<i'$, contradicting the fact that $x'\in L_{i',y'}$. 
  We conclude that the $\mathcal{L}_{i,y}$ are pairwise disjoint.
  For all $i\in[d]$ and $y$ not $i$-abundant, we have $|S\cap L_{i,y}|\le |S\cap P_{i,y}| \le h(\mathbb{N}^{d-1})$.
  By definition of $h(\cdot)$ it is possible to pick a point on each line in $\mathcal{L}_{i,y}$ to cover $S\cap L_{i,y}$.
  For each line in $\mathcal{L}_{0,0}$ in the $i$th direction, we pick the unique point $x\in[d]^d$ on the line such that $\sum_{j=0}^{d} x_{j} \equiv i\mod d$.
  This picks at most 1 point on each line and covers all elements of $S$, completing the proof.
\end{proof}

\vspace{3mm}
\noindent {\bf Acknowledgments.} We are grateful to Yuval Wigderson and Jacob Fox for stimulating conversations and helpful comments on this manuscript.

\bibliographystyle{alpha}

\end{document}